\documentclass{amsart}
\usepackage{graphicx} 
\usepackage{url, amsthm, amsmath}
\usepackage{amssymb}
\usepackage{appendix}

\newcommand{\boldA}{\mathbf{A}}
\newcommand{\boldB}{\mathbf{B}}
\newcommand{\boldC}{\mathbf{C}}
\newcommand{\boldD}{\mathbf{D}}
\newcommand{\boldE}{\mathbf{E}}

\theoremstyle{plain}
\newtheorem{theorem}{Theorem}[section]
\newtheorem{lemma}[theorem]{Lemma}
\newtheorem{fact}[theorem]{Fact}

\newtheorem{corollary}[theorem]{Corollary}

\newtheorem{proposition}[theorem]{Proposition}

\theoremstyle{definition}
\newtheorem{definition}[theorem]{Definition}
\newtheorem{remark}[theorem]{Remark}

\newcommand{\cstar}{$\mathrm{C}^*$}

\title{Non-computability of $K$-theory for computably presented \cstar-algebras}

\author[C. J. Eagle]{Christopher J. Eagle${}^1$} 
\address[C. J. Eagle]
{Department of Mathematics and Statistics, University of Victoria. PO BOX 1700 STN CSC, Victoria, British Columbia, Canada. V8W 2Y2}%
\email{eaglec@uvic.ca}
\urladdr{http://www.math.uvic.ca/~eaglec}
\thanks{${}^1$ Supported by NSERC Discovery Grant RGPIN-2021-02459}

\author[I. Goldbring]{Isaac Goldbring${}^2$}
\address[I. Goldbring]{University of California, Irvine}
\email{igoldbri@uci.edu}
\urladdr{https://www.math.uci.edu/~isaac}
\thanks{${}^2$  Supported by NSF grant DMS-2054477.}

\author[T. H. McNicholl]{Timothy H. McNicholl}
\address[T. H. McNicholl]{Iowa State University}
\email{mcnichol@iastate.edu}
\urladdr{https://faculty.sites.iastate.edu/mcnichol/}

\author[R. Miller]{Russell Miller${}^3$}
\address[R. Miller]{Queens College, 65-30 Kissena Blvd, Queens, NY 11367 \& CUNY Graduate Center, 365 Fifth Avenue, New York, NY 10016, U.S.A.}
\email{russell.miller@qc.cuny.edu}
\urladdr{https://qcpages.qc.cuny.edu/~rmiller/}
\thanks{${}^3$  Supported by NSF grant DMS-2348891,
by Simons Foundation award \#MP-TSM- 00007933,
and by several grants from the City University of New York PSC-CUNY Research Award Program.\\}

\begin{document}

\begin{abstract}
We give an example of a unital \cstar-algebra $\boldA$ with a computable presentation and for which neither $K_0(\boldA)$ nor $K_1(\boldA)$ has a computable presentation.
\end{abstract}

\maketitle

\section*{Introduction}
In \cite{UHFPaper}, we initiated the study of the effective content of the construction of the K-theory of a \cstar-algebra.  In particular, we constructed a computable functor that takes as input a c.e. presentation of a \cstar-algebra $\boldA$ and outputs a c.e. presentation of the group $K_0(\boldA)$ and another such functor outputting a c.e. presentation of $K_1(\boldA)$.  In some cases, such as when $\boldA$ is an AF algebra, our functor maps computable presentations of $\boldA$ to computable presentations of $K_0(\boldA)$; see \cite[Main Theorem 3]{AFPaper}.  This naturally raises the question of whether there is a \cstar-algebra $\boldA$ such that $\boldA$ has a computable presentation but $K_0(\boldA)$ does not (see \cite[Question 2.17]{UHFPaper}).  In this note, we give an example of a \cstar-algebra $\boldA$ such that $\boldA$ has a computable presentation but neither $K_0(\boldA)$ nor $K_1(\boldA)$ has a computable presentation.  We refer the reader to \cite{UHFPaper} for background on computability theory for \cstar-algebras and the computable $K_0$ and $K_1$ functors.

\section{A group with no computable presentation}
Fix, for the remainder of the note, a set $R \subseteq \mathbb{N}$ that is c.e. but not computable (for example, we can take $R = 0'$, the halting problem).  Enumerate the prime numbers in order as $(p_n)_{n \in \mathbb{N}}$.

For each $n \in \mathbb{N}$, set
\[G_n = \begin{cases}\mathbb{Z} & n \in R \\ \mathbb{Z}/p_n\mathbb{Z} & n \not\in R\end{cases}\]

For the rest of the note, set $G = \bigoplus_{n \in \mathbb{N}}G_n$.

\begin{proposition}\label{prop:GNoCompPres}
There is no computable presentation of the group $G$.
\end{proposition}
\begin{proof}
Suppose, towards a contradiction, that $G^\#$ is a computable presentation of $G$.  Let $P = \{n \in \mathbb{N} : G \text{ contains an element of order }p_n\}$.  Note that $G$ has an element of order $p_n$ if and only if $G_n = \mathbb{Z}/p_n\mathbb{Z}$ if and only if $n \not\in R$, that is, $P = \mathbb{N} \setminus R$.  

Since $G^\#$ is computable we can use it to successively compute the orders of each element of $G$, whence $P$ is a c.e. set.  Since $R$ was chosen to be a c.e. set,  this implies that $R$ is a computable set, contradicting our choice of $R$.
\end{proof}

\begin{remark}
It is not too difficult to show that the above group $G$ admits a c.e. presentation.  However, using \cite[Corollary 2.16]{UHFPaper}, this will be a consequence of the fact that $G$ is isomorphic to $K_0(\boldA)$ for a computably presentable \cstar-algebra $\boldA$.
\end{remark}

\section{Non-computable $K_0$}\label{sec:NonCompK0}
For each $n \in \mathbb{N}$, set
\[\boldB_n = \begin{cases}\mathbb{C} & n \in R \\ \mathcal{O}_{p_n+1} & n \not\in R\end{cases}\]
Here, $\mathcal{O}_k$ is the Cuntz algebra on $k$ generators, that is, $\mathcal{O}_k$ is the universal \cstar-algebra generated by partial isometries $s_1, \ldots, s_k$ such that each $s_i^*s_i=1$ and $\sum_{i=1}^ks_is_i^* = 1$.  It is well-known that each $\mathcal{O}_k$ is simple (\cite[Remarks after Theorem 1.12]{Cuntz1}) and has $K_0(\mathcal{O}_k) \cong \mathbb{Z}/(k-1)\mathbb{Z}$ (\cite[Theorem 3.7]{Cuntz2}).  We will also need the fact that the standard presentations of the Cuntz algebras are computable, uniformly in $n$; see \cite[Corollary 3.16]{fox2022computable}.

\begin{lemma}\label{lem:CompPresentBn}
The algebras $\boldB_n$ have computable presentations, uniformly in $n$.
\end{lemma}
\begin{proof}
We begin by noting that it suffices to show that each $\boldB_n$ has a c.e. presentation, uniformly in $n$.  Indeed, by a result of Fox \cite[Theorem 3.14]{fox2022computable}, any c.e. presentation of a \emph{simple} \cstar-algebra is computable (and the proof shows that the computability of the presentation is indeed uniform in the c.e. presentation).

Towards this end, fix $n$ and fix a computable enumeration $(R_s)_{s\in \mathbb{N}}$ of $R$, that is, $R_s$ consists of those natural numbers that have been determined to belong to $R$ by stage $s$.  Let $q_k$ denote the $k$th special point of the standard presentation of $\mathcal{O}_{p_n+1}$.  For all $k,s\in \mathbb{N}$ with $s\geq k$, we define $p_{k,s}\in \boldB_n$ as follows:
\begin{itemize}
    \item If $n\notin R_s$, set $p_{k,s}:=q_k$.
    \item If $n\in R_s\setminus R_k$, set $p_{k,s}:=0$.
    \item If $n\in R_k$, set $p_{k,s}:=1$.
\end{itemize}
If $s<k$, we declare that $p_{k,s}$ is undefined.  Given $k\in \mathbb{N}$, note that $p_{k,s}$ is eventually defined and constant; we denote this eventual constant value by $p_k$ and we set $Q:=\{p_k \ : \ k\in \mathbb{N}\}$.

Next note that in the case that $n\notin R$, we have $p_k=q_k$ for all $k$.  On the other hand, if $n\in R$ and $s_0\in \mathbb{N}$ is the least $s$ for which $n\in R_s$, then we have $p_k=0$ for $k<s_0$ and $p_k=1$ for $k\geq s_0$.  As a result, in either case, $Q$ generates a dense $*$-subalgebra of $\boldB_n$, whence we may consider $p_k$ as the $k$th special point of a presentation $\boldB_n^\#$ of $\boldB_n$.  

It remains to observe that $\boldB_n^\#$ is a c.e. presentation of $\boldB_n$, uniformly in $n$.  To see this, consider a rational $*$-polynomial $\rho(x_0,\ldots,x_m)$.  Let $(r_s)_{s\in \mathbb{N}}$ be a computable nonincreasing sequence converging to $\|\rho(q_0,\ldots,q_m)\|$ in $\mathcal{O}_{p_{n+1}}$.  We define a sequence $(t_s)_{s\geq m}$ as follows.  If $n\notin A_s$, set $t_s:=r_s$.  If $n\in R_s\setminus R_k$, set $t_s:=\rho(\vec 0)$.  If $n\in R_k$, set $t_s:=\rho(\vec 1)$.  It follows that $(t_s)_{s\in \mathbb{N}}$ is a computable nonincreasing sequence, uniformly in $n$, which converges to $\rho(p_0,\ldots,p_m)$, completing the proof.
\end{proof}

Set $\boldB = \bigoplus_{n \in \mathbb{N}}\boldB_n$.

\begin{proposition}\label{prop:Bcomp}
$\boldB$ has a computable presentation.
\end{proposition}
\begin{proof}
Using Lemma \ref{lem:CompPresentBn}, fix presentations $\boldB_n^\#$ for the algebras $\boldB_n$ such that these presentations are computable uniformly in $n$.  For each $N$, set $\boldB_N = \bigoplus_{n=1}^N\boldB_n$.  Then each $\boldB_N$ has a computable presentation $\boldB_N^\#$, uniformly in $N$, obtained by taking the direct sum of the presentations $\boldB_1^\#, \ldots, \boldB_N^\#$.  The inductive limit presentation constructed from the presentations $\boldB_N^\#$ as described in \cite[Section 2]{Goldbring.2024+} is then a computable presentation of $\lim_{N}\boldB_N=\boldB$, by \cite[Lemma 2.7(3)]{Goldbring.2024+}.
\end{proof}

\begin{theorem}\label{prop:KTheoryB}
$K_0(\boldB) \cong G$ and $K_1(\boldB) = 0$.
\end{theorem}
\begin{proof}
If $n \in R$ then $\boldB_n = \mathbb{C}$, so $K_0(\boldB_n) = \mathbb{Z}$, while if $n \not\in R$ then $\boldB_n = \mathcal{O}_{p_n+1}$, so $K_0(\boldB_n) = \mathbb{Z}/p_n\mathbb{Z}$.  That is, $K_0(\boldB_n) = G_n$ for all $n$.

Applying the fact that $K_0$ commutes with direct sums and direct limits \cite[Proposition 4.3.4 and Theorem 6.3.2]{Rordam.Larsen.Laustsen.2000}, we have:
\begin{align*}
K_0(\boldB) &= K_0\left(\bigoplus_{n \in \mathbb{N}}\boldB_n\right) \\
&\cong \bigoplus_{n\in\mathbb{N}}K_0(\boldB_n) \\
&\cong \bigoplus_{n \in \mathbb{N}}G_n \\
&= G
\end{align*}
The proof is complete by Proposition \ref{prop:GNoCompPres}.
\end{proof}

Putting together the results so far, we have:

\begin{corollary}
$\boldB$ is a computably presentable \cstar-algebra and $K_0(\boldB)$ has no computable presentation.
\end{corollary}

\section{Non-computable $K_1$}
Using the algebra $\boldB$ from the previous section, let $\boldC$ be the suspension of $\boldB$, that is, $\boldC = S\boldB := C_0(0, 1) \otimes \boldB$.

\begin{proposition}\label{prop:CCompPres}
$\boldC$ has a computable presentation.
\end{proposition}
\begin{proof}
This follows immediately from Proposition \ref{prop:Bcomp} and \cite[Lemma 3.1]{UHFPaper}.
\end{proof}

\begin{proposition}\label{prop:KTheoryC}
$K_0(\boldC) \cong G$ and $K_1(\boldC) = 0$.
\end{proposition}
\begin{proof}
Since $\boldC = S\boldB$, we have $K_0(\boldC) \cong K_1(\boldB)$ and $K_1(\boldC) \cong K_0(\boldB)$ (see \cite[Theorem 10.1.3 and Corollary 11.3.1]{Rordam.Larsen.Laustsen.2000}).
\end{proof}

We therefore have:

\begin{corollary}
$\boldC$ is a computably presentable \cstar-algebra and $K_1(\boldC)$ has no computable presentation.
\end{corollary}

\section{The main examples}
Putting together the examples from the previous two sections, we get a single computably presentable \cstar-algebra where neither $K_0$ nor $K_1$ is computably presented.  

\begin{theorem}\label{thm:BothNoComp}
There is a computably presentable \cstar-algebra $\boldA$ such that neither $K_0(\boldA)$ nor $K_1(\boldA)$ are computably presentable.
\end{theorem}
\begin{proof}
Using the algebras $\boldB$ and $\boldC$ defined earlier, let $\boldA = \boldB \oplus \boldC$.  The algebra $\boldA$ has a computable presentation because $\boldB$ and $\boldC$ have computable presentations.

As mentioned above,  $K_0$ commutes with direct sums, so Propositions \ref{prop:KTheoryB} and \ref{prop:KTheoryC} yield that
\[K_0(\boldA) \cong K_0(\boldB) \oplus K_0(\boldC) \cong G \oplus 0 \cong G.\]
Likewise, $K_1$ commutes with direct sums (\cite[Proposition 8.2.6]{Rordam.Larsen.Laustsen.2000}), so we have
\[K_1(\boldA) \cong K_1(\boldB) \oplus K_1(\boldC) \cong 0 \oplus G \cong G.\]
By Proposition \ref{prop:GNoCompPres}, neither $K_0(\boldA)$ nor $K_1(\boldA)$ has a computable presentation.
\end{proof}

Finally, we reach our main result.

\begin{theorem}\label{main}
There is a computably presentable, \emph{unital}, nuclear \cstar-algebra $\boldD$ such that neither $K_0(\boldD)$ nor $K_1(\boldD)$ is computably presentable.
\end{theorem}
\begin{proof}
Let $\boldD$ be the unitization of the algebra $\boldA$ used in the proof of Theorem \ref{thm:BothNoComp}.  Note that since each $\boldB_n$ is nuclear, $\boldD$ is nuclear as well.  Since $\boldA$ has a computable presentation, $\boldD$ does as well by \cite[Proposition 2.21]{UHFPaper}.  We have $K_1(\boldD) \cong K_1(\boldA)$ (\cite[Equation 8.4]{Rordam.Larsen.Laustsen.2000}), so $K_1(\boldD)$ has no computable presentation.

For $K_0$, we have $K_0(\boldD) \cong K_0(\boldA) \oplus \mathbb{Z}$ (\cite[Example 4.3.5]{Rordam.Larsen.Laustsen.2000}).  Since $G \oplus \mathbb{Z} \cong G$, $K_0(\boldD)$ has no computable presentation.
\end{proof}

\section{A stably finite example}

The algebra $\boldD$ in the previous section is \emph{infinite} in the sense that it has a proper isometry.  As mentioned in the introduction, AF algebras have the property that any computable presentation of them yields a computable induced presentation on the $K_0$ group.  Since AF algebras are finite, one might suspect that the distinction at play here is the finite vs. infinite one.  In this section, we show that one can in fact find stably finite \cstar-algebras satisfying the conclusion of Theorem \ref{main}.

For each $n\geq 2$, let $\Gamma_n:=\langle a,b \ | \ bab^{-1}=a^n\rangle$, the so-called Baumslag-Solitar group $BS(1,n)$.  The following lemma captures the key facts we need about the group \cstar-algebras of the Baumslag-Solitar groups.

\begin{lemma}
For every $n \geq 2$, $C^*(\Gamma_n)$ is a stably finite residually finite dimensional \cstar-algebra and has $K_0(C^*(\Gamma_n)) \cong \mathbb{Z}$ and $K_1(C^*(\Gamma_n)) \cong \mathbb{Z} \oplus \mathbb{Z}/(n-1)\mathbb{Z}$.
\end{lemma}
\begin{proof}
We may view $\Gamma_n = \langle a, b \ | \ bab^{-1} = a^n\rangle$ as a semi-direct product $\mathbb{Z}[1/n]\rtimes\mathbb{Z}$, where the $\mathbb{Z}$ factor corresponds to the subgroup $\langle b \rangle$ which acts by taking $n$th powers.  Thus $\Gamma_n$ is an extension of $\mathbb{Z}[1/n]$ by $\mathbb{Z}$, and hence is an (elementary) amenable group.  Therefore $C^*(\Gamma_n) \cong C^*_r(\Gamma_n)$ (see \cite[Theorem 2.6.8]{BrownOzawa}).  It thus suffices to show that $C^*_r(\Gamma_n)$ is residually finite dimensional and stably finite.  The latter property is shared by all reduced group \cstar-algebras, because they have faithful tracial states (see \cite[Proposition 2.5.3]{BrownOzawa}), so we only need to show that $C^*_r(\Gamma_n)$ is residually finite dimensional.

Each $\Gamma_n$ is a linear group, as it is isomorphic to the subgroup $$\left\{\begin{bmatrix}n^k & s \\ 0 & 1\end{bmatrix} : k \in \mathbb{Z}, s \in \mathbb{Z}[1/n]\right\}$$ of $\operatorname{GL}_2(\mathbb{C})$.  Thus by \cite[Theorem 4.3]{BekkaLouvet} each $C^*_r(\Gamma_n)$ is residually finite dimensional.  Again, we have $C^*(\Gamma_n) \cong C^*_r(\Gamma_n)$, so the former is also residually finite dimensional.
The statements about the $K$-theory of $C^*(\Gamma_n)$ are proved in \cite[Theorem 1]{PooyaValette}.
\end{proof}

Now define $\boldB_n$ as in Section \ref{sec:NonCompK0}, except replace $\mathcal{O}_{p_n+1}$ by $C^*(\Gamma_{p_n+1})$ when $n\notin R$.  Equip $C^*(\Gamma_{p_n+1})$ with its standard presentation.  Now one observes that Lemma \ref{lem:CompPresentBn} continues to hold in this setting.  Indeed, the standard presentation of $C^*(\Gamma_{p_n+1})$ is clearly c.e. uniformly in $n$.  Moreover, this presentation is in fact computable uniformly in $n$ by \cite[Theorem 3.5]{fox2022computable}, since $C^*(\Gamma_{p_n+1})$ is residually finite dimensional, thereby establishing Lemma \ref{lem:CompPresentBn} with the redefined $\boldB_n$.  As a result, setting $\boldB:=\bigoplus_n \boldB_n$, we have that $K_1(\boldB)\cong G$, where $G$ is the group admitting no computable presentation defined in Section 1, and $K_0(\boldB) = \bigoplus_n \mathbb{Z}$.  The algebra $\boldB$ is stably finite because each $\boldB_n$ is stably finite.  We note also that since each $\Gamma_n$ is amenable each $\boldB_n$ is nuclear, and hence $\boldB$ is nuclear as well.

We have thus found a stably finite \cstar-algebra $\boldB$ that is computably presentable (the validity of Proposition \ref{prop:Bcomp} depended only on the conclusion of Lemma \ref{lem:CompPresentBn}) for which $K_1(\boldB)$ is not computably presentable.  Setting $\boldA:=\widetilde{S\boldB}$, the unitization of the suspension of $\boldB$, we have that $K_0(\boldA)\cong K_1(\boldB)\oplus \mathbb Z\cong G\oplus \mathbb Z \cong G$.  Note also that $\boldA$ is computably presentable, stably finite and that $K_1(\boldA)\cong K_0(\boldB)=\bigoplus_n \mathbb{Z}$.  Finally, setting $\boldE:=\boldA\oplus \boldB$ yields a computably presentable, stably finite \cstar-algebra with $K_0(\boldE)\cong K_0(\boldA)\oplus K_0(\boldB) \cong G$ and $K_1(\boldE)\cong K_1(\boldA) \oplus K_1(\boldB) \cong G$, neither of which are computably presentable.  In summary, we have shown:

\begin{theorem}
There is a computably presentable, stably finite, unital, nuclear \cstar-algebra $\boldE$ such that neither $K_0(\boldE)$ nor $K_1(\boldE)$ has a computable presentation.
\end{theorem}

\section*{Acknowledgments} \label{sec:ack}
The results of this note were obtained during the authors' visit to the American Institute of Mathematics (AIM) in January 2026 as part of their SQuaREs program.  The authors thank AIM for their hospitality and for creating a wonderful working environment during the writing of this paper.

\bibliographystyle{amsplain}
\bibliography{paperbib}
\end{document}